\documentclass{amsart}
\usepackage[utf8]{inputenc}
\usepackage{amssymb, amsmath, amsthm}
\usepackage{color}
\usepackage{cite}
\newtheorem{theorem}{Theorem}[section]
\newtheorem*{theorem*}{Theorem}
\newtheorem{lemma}{Lemma}

\newtheorem{corollary}{Corollary}
\theoremstyle{remark}
\newtheorem{rem}{Remark}

\usepackage{mathtools}
\DeclareMathOperator{\rank}{rank}
\newcommand{\trans}{\mathsf{T}}

\newcommand{\C}{\mathbb{C}}
\newcommand{\sz}[2]{#1\times #2}

\newcommand{\A}{\mathbf{A}}
\definecolor{min}{rgb}{0,.5,.5}
\newcommand{\pen}{D}
\newcommand{\minu}{\mathbf{e}}
\newcommand{\minv}{\mathbf{f}}
\newcommand{\invpart}{\mathbf{G}}
\newcommand{\minw}{\mathbf{x}}
\newcommand{\minx}{\mathbf{y}}

\newcommand{\rang}{r}
\newcommand{\red}{r}
\newcommand{\rest}{k}

\newcommand{\stmp}{\mathbf u}
\newcommand{\ttmp}{\mathbf v}
\newcommand{\EE}{M}
\newcommand{\FF}{N}

\title{Inverting the sum of two singular matrices}
\author{Sofia Eriksson${}^{*,\dagger}$ and Jonas Nordqvist${}^{*}$}

\begin{document}

\begin{abstract}
Square matrices of the form~$\widetilde{\A} =\A + \minu\pen \minv^*$ are considered. An explicit expression for the inverse is given, provided~$\widetilde{\A}$ and~$\pen$ are invertible with~$\rank(\widetilde{\A}) =\rank(\A)+\rank(\minu\pen \minv^*)$. The inverse is presented in two ways, one that uses singular value decomposition and another that depends directly on the components~$\A$,~$\minu$,~$\minv$ and~$\pen$. Additionally,  a matrix determinant lemma for singular matrices follows from the derivations.
\end{abstract}

\maketitle

{\tiny

* Department of Mathematics, Linnaeus University, Växjö, Sweden

\vspace{-4pt}

$\dagger$ Corresponding author: sofia.eriksson@lnu.se
}

\ 

{\small


}
\section{Introduction}

We are interested in the inverse~$(A+B)^{-1}$, where~$B$ is seen as an update or a perturbation to~$A$.
Notably among previous works is the Woodbury matrix identity 
(also called the Sherman--Morrison--Woodbury formula) as detailed \emph{e.g.} in~\cite{Woodbury1950}:
\begin{align}\label{Woodbury}
 \left(A+UC V\right)^{-1}=A^{-1}-A^{-1}U\left(C^{-1}+V A^{-1}U\right)^{-1}V A^{-1}.
 \end{align}
%However, in some applications~$A$ is singular and then \eqref{Woodbury} is inapplicable. 
%For example,
However, \eqref{Woodbury} is inapplicable if~$A$ is singular. As an example of application,
matrices originating from finite difference schemes are inherently singular
and have to be modified using boundary conditions to become invertible.
In particular, for finite difference schemes referred to as SBP-SAT, the modified matrix is of the form 
$A+B$ where both~$A$ and~$B$ are singular but the sum~$A+B$ usually nonsingular. 
This specific type of matrices were inverted explicitly in~\cite{Eriksson20092659, InversesEriksson2021} with an extension to pseudoinverses in~\cite{doi:10.1137/20M1379083}. See~\cite{Magnus201417,DelReyFernandez2014171} for a background on the SBP-SAT methods. %Matrices, originating from this specific type of method, were inverted explicitly in~\cite{Eriksson20092659, InversesEriksson2021} with an extension to pseudoinverses in~\cite{doi:10.1137/20M1379083}.

In general, given two singular matrices~$A$ and~$B$, it is well-known that there exist instances such that~$\rank(A+B) = \rank(A) + \rank(B)$ (see \emph{e.g.}~\cite{marsaglia_styan_1972} for a treatment on the topic). In this paper, we focus on the case where~$A$ and~$B$ are square~$n\times n$ rank-deficient matrices with $\rank(A)=n-k$ and~$\rank(B)=k$. 
Specifically, we are interested in matrices of the form \[\widetilde{\A} =\A + \minu\pen \minv^*,\]
and our main result, presented in Theorem~\ref{thmsvd}, is an explicit inversion formula of~$\widetilde{\A}$.\footnote{By the notation $\minv^*$ we mean the conjugate transpose of $\minv$.} The inversion is stated in terms of the matrices~$\invpart$,~$\minw$ and~$\minx$, neither of which depend on~$\pen$. The explicit formula reads
\[\widetilde{\A}^{-1}=\invpart+\minw\pen^{-1}\minx^*,\] where ~$\invpart$,~$\minw$ and~$\minx$ are given in \eqref{MWZsvdparts}, and in alternative forms in \eqref{WZ}.

Previous works have explored similar problem formulations in various contexts.
Extensions of the Woodbury formula \eqref{Woodbury} in terms of generalized inverses have been presented in several papers, for example~\cite{Riedel1992, Deng2011, DongweiJianbing2020}. 
The works by~\cite{HendersonSearle1981, Miller81} provide further related insights on the topic of inverting sums of matrices.
The expressions proven in~\cite{Deng2011, DongweiJianbing2020} remind closely of \eqref{Woodbury} except~$A^{-1}$ is replaced by different types of generalized inverses. The approach in~\cite{Riedel1992} offers results most akin to ours, as discussed further in Remark \ref{riedel}.

\section{The inverse of the sum of two singular matrices}

We  consider a matrix~$\widetilde{\A}$ of the form
\begin{align}\label{ourB}
\widetilde{\A}=\A+\minu\pen \minv^*
\end{align}
where~$\A$,~$\minu$,~$\pen$ and~$\minv$ all denote conformable matrices.  
The matrices~$\widetilde{\A}$ and~$\pen$ are both invertible but~$\A$ and~$\minu\pen \minv^*$ are singular.\footnote{
In the previously mentioned applications, typically~$\A$ is having high rank and the product~$\minu\pen \minv^*$ having low rank.} 
Based on observations from those cases in~\cite{InversesEriksson2021} that satisfy the condition~$\rank(A)+\rank(B) = \rank(A+B) = n$, we expect to find a  formula for~$\widetilde{\A}^{-1}$ of the form~$\widetilde{\A}^{-1}=\invpart+\minw\pen^{-1}\minx^*$, where~$\invpart$,~$\minw$ and~$\minx$ do not  depend on~$\pen$.
We first use singular value decomposition to confirm this expectation and derive the components of~$\widetilde{\A}^{-1}$ yielding the following theorem:

\begin{theorem}\label{thmsvd}
Let~$n,k$ be integers such that~$n>k\geq1$. Consider the complex matrix~$\widetilde{\A}=\A+\minu\pen \minv^*,$ where~$\A, \minu$,~$\pen$ and~$\minv$ are~$\sz{n}{n}$,~$\sz{n}{k}$,~$\sz{k}{k}$ and~$\sz{n}{k}$ respectively.
Further, suppose the rank of~$\A$ is~$\rang \coloneqq n-k$. It is assumed that the 
columns of~$\minu$ together with the columns of~$\A$ span~$\C^n$. Similarly, we assume that 
the
columns of~$\minv$ 
span~$\C^n$ together with~$\A^*$. Now, given that~$\pen$ is invertible,
then so is~$\widetilde{\A}$, and its inverse is
\begin{align}\label{invstruct}\widetilde{\A}^{-1}=\invpart+\minw\pen^{-1}\minx^*\end{align}
with
$\invpart$,~$\minw$ and~$\minx$ given in \eqref{MWZsvdparts}.
\end{theorem}

\begin{proof}[Proof of Theorem~\ref{thmsvd}]

Let~$\A =\mathbf {U\Sigma V^{*}} =U_{\red}\Sigma_{\red}V_{\red}^*$ be a singular value decomposition of~$\A$ such that~$\mathbf {U, V}$ are complex unitary~$n\times n$-matrices and such that the latter expression refers to a compact singular value decomposition. That is~$U_{\red}$,~$V_{\red}$ are semi-unitary~$n\times \rang$-matrices, such that ~$ U_{\red}^* U_{\red} =I _\rang$ and~$V_{\red}^* V_{\red}=I _\rang$, where~$I _\rang$ is the~$\rang\times\rang$ identity matrix. The relation between the full and the compact singular value decomposition is written in block matrix form as
\begin{align}\label{BUSV}
\A &=\mathbf {U\Sigma V^{*}} =\left[\begin{array}{cc}U_{\red} &U_{\rest}\end{array}\right]\left[\begin{array}{cc}\Sigma_{\red}&0\\0&0\end{array}\right]\left[\begin{array}{c}V_{\red}^* \\V_{\rest}^*\end{array}\right]=U_{\red}\Sigma_{\red}V_{\red}^*,
\end{align}
where~$U_{\rest}$ refers to the~$k$ 
columns of~$\mathbf{U}$ that are not included in~$U_{\red}$ and where~$V_{\rest}$ refers to the~$k$ columns of~$\mathbf{V}$ that are not included in~$V_{\red}$.  
We can write~$\widetilde{\A}$ 
as
\begin{align}\label{eq:svd} \widetilde{\A}=\A+\minu\pen\minv^*=\left[\begin{array}{cc}U_{\red} &\minu\end{array}\right]\left[\begin{array}{cc}\Sigma_{\red}&0\\0&\pen\end{array}\right]\left[\begin{array}{c}V_{\red}^* \\\minv^*\end{array}\right].
\end{align}
Since it is given that the columns of~$\A$ together with~$\minu$ span~$\C^n$,~$U_{\red}$ together with~$\minu$ must form a basis. Thus~$\left[\begin{array}{cc}U_{\red} &\minu\end{array}\right]$ is a square, invertible matrix. Similarly, ~$\minv$  spans~$\C^n$ together with~$\A^*$ making~$\left[\begin{array}{cc}V_{\red} &\minv\end{array}\right]$ invertible. Thus we can write
\begin{align}
\label{AinvSVD}
\widetilde{\A}^{-1}=\left[\begin{array}{c}V_{\red}^* \\\minv^*\end{array}\right]^{-1}\left[\begin{array}{cc}\Sigma_{\red}^{-1}&0\\0&\pen^{-1}\end{array}\right]\left[\begin{array}{cc}U_{\red} &\minu\end{array}\right]^{-1}.
\end{align}
Next, using the orthogonality properties of~$\mathbf U$, that is that~$ U_{\red}^* U_{\red} =I _\rang$ and~$U_{\rest}^* U_{\red}=0$, we find 
\begin{align}\label{UkTe} 
\left[\begin{array}{c}U_{\red}^* \\U_{\rest}^*\end{array}\right]\left[\begin{array}{cc}U_{\red} &\minu\end{array}\right]=\left[\begin{array}{cc}U_{\red}^* U_{\red} &U_{\red}^*\minu\\U_{\rest}^* U_{\red} &U_{\rest}^*\minu\end{array}\right]
=\left[\begin{array}{cc}I_\rang&U_{\red}^*\minu\\0&U_{\rest}^*\minu\end{array}\right],\end{align}
which using inverses of block matrices leads to
\begin{align*}
\left[\begin{array}{cc}U_{\red} &\minu\end{array}\right]^{-1}&=\left[\begin{array}{cc}I_\rang&U_{\red}^*\minu\\0&U_{\rest}^*\minu\end{array}\right]^{-1}\left[\begin{array}{c}U_{\red}^* \\U_{\rest}^*\end{array}\right]\\
&=\left[\begin{array}{cc}I_\rang&-U_{\red}^*\minu(U_{\rest}^*\minu)^{-1}\\0&(U_{\rest}^*\minu)^{-1}\end{array}\right]\left[\begin{array}{c}U_{\red}^* \\U_{\rest}^*\end{array}\right]\\
&=\left[\begin{array}{c}U_{\red}^*-U_{\red}^*\minu(U_{\rest}^*\minu)^{-1}U_{\rest}^*\\(U_{\rest}^*\minu)^{-1}U_{\rest}^*\end{array}\right].
\end{align*}
Correspondingly, the orthogonality properties of~$\mathbf V$ gives
\begin{align*}
\left[\begin{array}{c}V_{\red}^*\\\minv^*\end{array}\right]^{-1}&=\left[\begin{array}{cc}V_{\red}-V_{\rest}(\minv^* V_{\rest})^{-1}\minv^* V_{\red}&V_{\rest}(\minv^* V_{\rest})^{-1}\end{array}\right].
\end{align*}
Inserting these relations  into \eqref{AinvSVD} 
gives~$\widetilde{\A}^{-1}=\invpart
+
\minw\pen^{-1}\minx^*$ with 
\begin{align}\label{MWZsvdparts}
 \begin{split}
\invpart&=(V_{\red}-V_{\rest}(\minv^* V_{\rest})^{-1}\minv^* V_{\red})\Sigma_{\red}^{-1}(U_{\red}^*-U_{\red}^*\minu(U_{\rest}^*\minu)^{-1}U_{\rest}^*)\\
\minw&=V_{\rest}
(\minv^* V_{\rest})^{-1}\\
\minx&=U_{\rest}(\minu^* U_{\rest})^{-1}.
\end{split}
\end{align}
The proof is completed by noting that the existence of~$(\minu^* U_{\rest})^{-1}$ follows from \eqref{UkTe}, 
and analogously for~$(\minv^* V_{\rest})^{-1}$. 
\end{proof}

The special structure of~$\widetilde{\A}^{-1}$ in \eqref{invstruct} leads to the following relations:
\begin{corollary}[Corollary of Theorem~\ref{thmsvd}]
\label{cor1}

Consider~$\widetilde{\A}=\A+\minu\pen \minv^*$ in \eqref{ourB}, and its inverse
$\widetilde{\A}^{-1}=\invpart+\minw\pen^{-1}\minx^*$.
The components of~$\widetilde{\A}$ and~$\widetilde{\A}^{-1}$ satisfy
\begin{align}\label{prop1ny}
\A\minw=0,&&
\minx^* \A=0,&&
\invpart \minu=0&&
\minv^* \invpart=0&&
\minv^* \minw=I_k&&
\minx^* \minu=I_k,
\end{align}
where~$I_k$ is the~$k\times k$ identity matrix.
Moreover,
\begin{align}\label{prop2ny}
I_n&=\A\invpart+\minu \minx^* ,&
I_n&=\invpart \A+\minw\minv^*.
\end{align}
\end{corollary}
\begin{proof}[Proof of Corollary~\ref{cor1}]
The first two relations in \eqref{prop1ny} are easily verified by insertion of 
~$\A =U_{\red}\Sigma_{\red}V_{\red}^*$ from \eqref{BUSV}
as well as the expressions for
$\minw$ and 
$\minx^*$ from
 \eqref{MWZsvdparts}  and thereafter using the orthogonality properties~$V_{\red}^* V_{\rest}=0$ and~$U_{\rest}^* U_{\red}=0$.
The latter four relations in \eqref{prop1ny} are trivially found using \eqref{MWZsvdparts}.

The relations in \eqref{prop2ny} follow from
inserting the expressions 
in
\eqref{ourB} and  
\eqref{invstruct}  into the identities~$I_n=\widetilde{\A}\widetilde{\A}^{-1}=\widetilde{\A}^{-1}\widetilde{\A}$, 
and thereafter using \eqref{prop1ny}.
\end{proof}

\subsection{Relations to generalized inverses}

Generalized inverses are sometimes classified using the Penrose conditions
\begin{align*}
i)\ AA^gA=A,&&
ii)\ A^gAA^g=A^g,&&
iii)\ (AA^g)^*=AA^g,&&
iv)\ (A^gA)^*=A^gA,
\end{align*}
and if all four properties are fulfilled then~$A^g$ is the Moore--Penrose inverse~\cite{Wangguorong2018} denoted~$A^+$.
From \eqref{prop2ny}, it is evident that~$\invpart$ times the singular matrix~$\A$ is almost equal to the identity matrix, suggesting that~$\invpart$ might be some kind of generalized inverse to~$\A$. 
Indeed, multiplying the relations in \eqref{prop2ny} by~$\A$ or~$\invpart$ and then using~\eqref{prop1ny}, we find that~$\A\invpart\A=\A$ and~$\invpart\A\invpart=\invpart$, \emph{i.e.}, the first two Penrose conditions are satisfied.
Condition ({\it iii}) and ({\it iv}) are only satisfied if
$\minu \minx^*$ and~$\minw\minv^*$ in \eqref{prop2ny} happen to be symmetric. Thus,~$\invpart\neq \A^+$  in general.
Instead, we note that since~$\A$ is a square matrix, its 
Moore--Penrose inverse is
\begin{align*}
 \A^+=V_{\red}\Sigma_{\red}^{-1}U_{\red}^*.
\end{align*}
Using this, we may express $\invpart$ from \eqref{MWZsvdparts} in terms of~$\A^+$ as
\begin{align}\label{relpseudo}
\invpart
&=(I_n-V_{\rest}(\minv^* V_{\rest})^{-1}\minv^* )\A^+(I_n-\minu(U_{\rest}^*\minu)^{-1}U_{\rest}^*).
\end{align}

\begin{rem}\label{riedel}
In~\cite{Riedel1992}, the Moore--Penrose inverse of~$\Omega =A+(V_1+W_1)G(V_2+W_2)^*$ is derived.
The matrix~$V_1$ is in the column space of~$A$, and~$W_1$ is orthogonal to it (and correspondingly for~$V_2$,~$W_2$ and~$A^*$). Using equation (2) in~\cite{Riedel1992}, given the existence of the inverses of the matrices~$\Omega$ and~$G$ therein, we obtain
\begin{align}\label{RiedelInv}
    \Omega^{-1}=(I_n-C_2V_2^*)A^+(I_n-V_1C_1^*)+C_2G^{-1}C_1^*,
\end{align}
where~$C_i=W_i(W_i^* W_i)^{-1}$. 
Comparing to
~$\widetilde{\A}=\A+\minu\pen\minv^*$, 
we identify 
\begin{align*}
\widetilde\A=\Omega,&&\A=A,&&
\minu=V_1+W_1,&&
\minv=V_2+W_2,&&
\pen=G.
\end{align*}
Relating the components of~\cite{Riedel1992} to our notation, the variables~$V_1$ and~$W_1$ can be obtained as the projections of~$\minu$ onto~$U_{\red}$ and~$U_{\rest}$, respectively (and correspondingly for~$V_2$,~$W_2$ of~$\minv$ onto~$V_{\red}$ and~$V_{\rest}$). Thus~$V_{1,2}$ and ~$C_{1,2}$ 
in \eqref{RiedelInv} correspond to
\begin{align*}V_1=U_{\red}U_{\red}^*\minu,&&
V_2=V_{\red}V_{\red}^*\minv,&&
C_1=U_{\rest}(\minu^* U_{\rest})^{-1},&& C_2=V_{\rest}(\minv^* V_{\rest})^{-1}.\end{align*}
We compare~$\Omega^{-1}$ to our results; using \eqref{MWZsvdparts}  and \eqref{relpseudo} we can rewrite~$\widetilde{\A}^{-1}$ in \eqref{invstruct} as
\begin{align*}   \widetilde{\A}^{-1}&=(I_n-V_{\rest}(\minv^* V_{\rest})^{-1}\minv^* )\A^+(I_n-\minu(U_{\rest}^*\minu)^{-1}U_{\rest}^*)\\&\quad+V_{\rest}
(\minv^* V_{\rest})^{-1}\pen^{-1}(U_{\rest}(\minu^* U_{\rest})^{-1})^*,
\end{align*}
which has a clear resemblance  to \eqref{RiedelInv}.
The formulas are not completely identical, since~$V_1C_1^*=U_{\red} U_{\red}^*\minu(U_{\rest}^* \minu )^{-1}U_{\rest}^*\neq\minu(U_{\rest}^*\minu)^{-1}U_{\rest}^*$, however
the difference belongs to the nullspace of~$\A^+=A^+$, such that
\begin{align*}
\A^+\minu(U_{\rest}^*\minu)^{-1}U_{\rest}^*&=\A^+(U_{\red}U_{\red}^*\minu+U_{\rest}U_{\rest}^*\minu)(U_{\rest}^*\minu)^{-1}U_{\rest}^*\\
&=\A^+U_{\red}U_{\red}^*\minu(U_{\rest}^*\minu)^{-1}U_{\rest}^*\\
&=A^+V_1C_1^*
\end{align*} since~$\A^+U_{\rest}=0$. 
Correspondingly, the difference between~$C_2V_2^*$ and~$ V_{\rest}(\minv^* V_{\rest})^{-1}\minv^*$ belongs to the nullspace of~$(\A^+)^*$, and hence 
$\widetilde{\A}^{-1}=\Omega^{-1}$.
\end{rem}

\subsection{Expressing the inverse without 
singular value decomposition}

In the paper~\cite{InversesEriksson2021}, 
special cases of these matrix sum inverses
were derived, 
where~$\A$ originated from finite difference stencils and %
$\minu$ and~$\minv$ depended on boundary conditions. The strategy of finding the inverse did not involve singular value decomposition. 
Instead, ideas from the relations \eqref{prop1ny} and \eqref{prop2ny} were used.

In an attempt to generalize those ideas, we modify~$\A$ and~$\invpart$, multiply the results and aim to obtain the identity matrix, using the yet unknown~$k\times k$ matrices~$\EE$,~$\FF$ and~$\sz{n}{k}$ matrices~$\stmp$,~$\ttmp$, in the following ansatz:
\begin{align}\label{eq:ansatz}
J_n&\coloneqq\left(\left(I_n-\minu(\stmp^* \minu)^{-1}\stmp^* \right)\A\left(I_n-\ttmp(\minv^* \ttmp)^{-1}\minv^* \right)+\minu \EE\minv^* \right)(\invpart+\ttmp \FF\stmp^* ).
\end{align}
We can rewrite $J_n$ as
\begin{align*}
J_n &=\underbrace{\left(I_n-\minu(\stmp^* \minu)^{-1}\stmp^* \right)\A\left(I_n-\ttmp(\minv^* \ttmp)^{-1}\minv^* \right)\invpart}_{I_n-\minu(\stmp^* \minu)^{-1}\stmp^*}+\minu \EE\underbrace{\minv^* \invpart}_{0}\\
&+\left(I_n-\minu(\stmp^* \minu)^{-1}\stmp^* \right)\A\underbrace{\left(I_n-\ttmp(\minv^* \ttmp)^{-1}\minv^* \right)\ttmp}_{0}\FF\stmp^* +\minu \EE\minv^* \ttmp \FF\stmp^* \\
&=I_n-\minu(\stmp^* \minu)^{-1}\stmp^* +\minu \EE\minv^* \ttmp \FF\stmp^*,
\end{align*}
where we have used~$\minv^* \invpart=0$ from \eqref{prop1ny} and~$I_n=\A\invpart+\minu \minx^*$ from \eqref{prop2ny} to simplify the expressions. We note that $J_n$ equals the identity matrix $I_n$ if~$(\stmp^* \minu)^{-1}=\EE\minv^* \ttmp \FF$.
Inserting~$\FF=(\minv^* \ttmp)^{-1}\EE^{-1}(\stmp^* \minu)^{-1}$ into the ansatz \eqref{eq:ansatz}, 
 we can solve for~$\invpart$.
This results in 
\begin{align*}
\invpart&=\left(\left(I_n-\minu (\stmp^* \minu )^{-1}\stmp^* \right)\A\left(I_n-\ttmp (\minv^* \ttmp )^{-1}\minv^* \right)+\minu \EE\minv^* \right)^{-1}\\
&\quad-\ttmp (\minv^* \ttmp )^{-1}\EE^{-1}(\stmp^* \minu )^{-1}\stmp^* ,
\end{align*}
where~$\EE$ is an arbitrary invertible~$k \times k$ matrix. The matrices~$\stmp$,~$\ttmp$ are arbitrary up to the requirement that~$(\minv^* \ttmp )^{-1}$ and~$(\stmp^* \minu )^{-1}$ must exist.  The simplest choice is probably~$\stmp=\minu$,~$\ttmp=\minv$, %yielding
% \begin{align*}
% \invpart&=\left(\left(I_n-\minu (\minu^* \minu )^{-1}\minu^* \right)\A\left(I_n-\minv (\minv^* \minv )^{-1}\minv^* \right)+\minu \EE\minv^* \right)^{-1}\\
% &\quad-\minv (\minv^* \minv )^{-1}\EE^{-1}(\minu^* \minu )^{-1}\minu^*,
% \end{align*}
and unless the conditioning gets bad (for numerical purposes),~$\EE=I$ should suffice, simplifying the expression
of $\invpart$
even further
% \begin{align*}
% \invpart&=\left(\left(I_n-\minu (\minu^* \minu )^{-1}\minu^* \right)\A\left(I_n-\minv (\minv^* \minv )^{-1}\minv^* \right)+\minu \minv^* \right)^{-1}\\
% &\quad-\minv (\minv^* \minv )^{-1}(\minu^* \minu )^{-1}\minu^*.
% \end{align*}
to what is presented below in \eqref{WZ}.

To obtain~$\minw$ and~$\minx$, we multiply the right expression in \eqref{prop2ny} by~$\ttmp (\minv^* \ttmp )^{-1}$ from the right and the left expression by~$(\stmp^* \minu )^{-1}\stmp^*~$ from the left, respectively. Note that the matrices~$\stmp$,~$\ttmp$ does not necessarily have to be the same as above, as before the only requirement is that~$(\minv^* \ttmp )^{-1}$ and~$(\stmp^* \minu )^{-1}$ exist. This yields the identities
\begin{align*}
\minw&=(I_n-\invpart \A)\ttmp (\minv^* \ttmp )^{-1},&
\minx^*&=(\stmp^* \minu )^{-1}\stmp^* (I_n-\A\invpart).
\end{align*}
If for simplicity choosing~$\stmp=\minu$ and~$\ttmp=\minv$ we obtain
% \begin{align*}
% \minw&=(I_n-\invpart \A)\minv (\minv^* \minv )^{-1},&
% \minx^*&=(\minu^* \minu )^{-1}\minu^* (I_n-\A\invpart).
% \end{align*}
% We summarize 
the (not fully general) expressions in
\begin{align}\label{WZ}\begin{split}
\invpart&=\left(\left(I_n-\minu (\minu^* \minu )^{-1}\minu^* \right)\A\left(I_n-\minv (\minv^* \minv )^{-1}\minv^* \right) +\minu \minv^* \right)^{-1}\\&\quad-\minv (\minv^* \minv )^{-1}(\minu^* \minu )^{-1}\minu^*,\\
\minw&=(I_n-\invpart \A)\minv (\minv^* \minv )^{-1},\\
\minx^*&=(\minu^* \minu )^{-1}\minu^* (I_n-\A\invpart).
\end{split}
\end{align}
In cases where $\minx$ and $\minw$ are known, one can instead put $\stmp=\minx$ and $\ttmp=\minw$ in the general $\invpart$ above, leading to $\invpart=\left(\A+\minu \EE\minv^* \right)^{-1}-\minw \EE^{-1}\minx^*$.  This should  come as no surprise, since \eqref{invstruct} holds for any invertible $\pen$, including $\pen=\EE$.

\subsection{Singular matrix determinant lemma}
For a real invertible matrix~$A$, the  \emph{matrix determinant lemma} (see \emph{e.g.}~\cite{Harville97}), in generalized form
$$ \det( A +UV^\trans)=\det(I +V ^\trans A ^{-1}U )\det(A ),$$
offers an explicit expression of the determinant of a perturbed matrix. In the case where~$\A$ is singular, though, the regular matrix determinant lemma is not applicable. However, we note that 
using the determinant of products and the determinant of block matrices on 
\eqref{eq:svd}, a singular version of the lemma follows: 

\begin{lemma}
Let $\widetilde{\A}$ be defined as in Theorem~\ref{thmsvd}. Then
\[\det(\widetilde{\A}) = \det(\A + \minu\minv^*)\det(\pen).\]
\end{lemma}
Thus, given the nonsingularity of~$\A  + \minu\minv^*$, the singularity of~$\widetilde{\A}$ is solely determined upon by~$\det(\pen)$. We further note that the inverse~$\widetilde{\A}^{-1}$  described in Theorem~\ref{thmsvd}  admit the analogous relation $\det(\widetilde{\A}^{-1})=\det(\invpart+\minw\minx^*)\det(\pen^{-1})$.

\bibliographystyle{alpha} 
%\bibliography{InvPertBib}
\bibliography{references}

\begin{thebibliography}{WWQ18}

\bibitem[Den11]{Deng2011}
C.~Y. Deng.
\newblock A generalization of the {S}herman-{M}orrison-{W}oodbury formula.
\newblock {\em Applied Mathematics Letters}, 24(9):1561--1564, 2011.

\bibitem[EN09]{Eriksson20092659}
S.~Eriksson and J.~Nordstr\"{o}m.
\newblock Analysis of the order of accuracy for node-centered finite volume
  schemes.
\newblock {\em Applied Numerical Mathematics}, 59(10):2659--2676, 2009.

\bibitem[Eri21]{InversesEriksson2021}
S.~Eriksson.
\newblock Inverses of {SBP}-{SAT} finite difference operators approximating the
  first and second derivative.
\newblock {\em Journal of Scientific Computing}, 89, 2021.

\bibitem[EW21]{doi:10.1137/20M1379083}
S.~Eriksson and S.~Wang.
\newblock Summation-by-parts approximations of the second derivative:
  Pseudoinverse and revisitation of a high order accurate operator.
\newblock {\em SIAM Journal on Numerical Analysis}, 59(5):2669--2697, 2021.

\bibitem[FHZ14]{DelReyFernandez2014171}
D.~C. Del~Rey Fern\'{a}ndez, J.~E. Hicken, and D.~W. Zingg.
\newblock Review of summation-by-parts operators with simultaneous
  approximation terms for the numerical solution of partial differential
  equations.
\newblock {\em Computers \& Fluids}, 95:171--196, 2014.

\bibitem[Har97]{Harville97}
D.~A. Harville.
\newblock {\em Matrix algebra from a statistician's perspective}.
\newblock Springer-Verlag, New York, 1997.

\bibitem[HS81]{HendersonSearle1981}
H.~V. Henderson and S.~R. Searle.
\newblock On deriving the inverse of a sum of matrices.
\newblock {\em SIAM Review}, 23(1):53--60, 1981.

\bibitem[Mil81]{Miller81}
K.~S. Miller.
\newblock On the inverse of the sum of matrices.
\newblock {\em Mathematics Magazine}, 54(2):67--72, 1981.

\bibitem[MS72]{marsaglia_styan_1972}
G.~Marsaglia and G.~P.~H. Styan.
\newblock When does rank({A}+{B}) = rank({A}) + rank({B})?
\newblock {\em Canadian Mathematical Bulletin}, 15(3):451–452, 1972.

\bibitem[Rie92]{Riedel1992}
K.~S. Riedel.
\newblock A {S}herman–{M}orrison–{W}oodbury identity for rank augmenting
  matrices with application to centering.
\newblock {\em SIAM Journal on Matrix Analysis and Applications},
  13(2):659--662, 1992.

\bibitem[SC20]{DongweiJianbing2020}
D.~Shi and J.~Cao.
\newblock The {S}herman-{M}orrison-{W}oodbury formula for the {M}oore-{P}enrose
  metric generalized inverse.
\newblock {\em Linear and Multilinear Algebra}, 68(5):972--982, 2020.

\bibitem[SN14]{Magnus201417}
M.~Sv\"{a}rd and J.~Nordstr\"{o}m.
\newblock Review of summation-by-parts schemes for initial-boundary-value
  problems.
\newblock {\em Journal of Computational Physics}, 268:17--38, 2014.

\bibitem[Woo50]{Woodbury1950}
M.~A. Woodbury.
\newblock {\em Inverting modified matrices}.
\newblock Princeton University, Princeton, NJ, 1950.
\newblock Statistical Research Group, Memo. Rep. no. 42,.

\bibitem[WWQ18]{Wangguorong2018}
G.~Wang, Y~Wei, and S.~Qiao.
\newblock {\em Generalized Inverses: Theory and Computations}.
\newblock Springer, 2018.

\end{thebibliography}

\end{document}